\documentclass[reqno,11pt]{amsart}
\usepackage{amssymb, bm, amsmath, latexsym, tikz, float}
\usepackage{makecell, multirow, longtable}
\usepackage{enumitem}
\usepackage{mathrsfs}
\usepackage{fullpage}

\usepackage{hyperref}

\numberwithin{equation}{section}

\newtheorem{thm}{Theorem}[section]
\newtheorem{lem}[thm]{Lemma}

\newtheorem{prop}[thm]{Proposition}
\newtheorem{conj}[thm]{Conjecture}
 
\theoremstyle{definition}
\newtheorem{defn}[thm]{Definition}

\theoremstyle{remark}
\newtheorem{rmk}[thm]{Remark}
\newtheorem{exam}[thm]{Example}

\numberwithin{equation}{section}

\newcommand{\ra}{{\, \rightarrow \,}}

\newcommand{\z}{{\mathbb Z}}

\newcommand{\q}{{\mathbb Q}}

\newcommand{\n}{{\mathbb N}}

\newcommand{\oo}{{\mathcal{O}}}
\newcommand{\ok}[1]{{\mathcal{O}_K^{#1}}}
\newcommand{\of}[1]{{\mathcal{O}_F^{#1}}}
\newcommand{\Mod}[1]{\ (\mathrm{mod}\ #1)}

\newcommand{\tr}{\operatorname{Tr}}

\usepackage[colorinlistoftodos]{todonotes}

\renewcommand{\r}{\mathbb{R}}

\newcommand{\Z}{\mathbb{Z}}
\newcommand{\Q}{\mathbb{Q}}
\newcommand{\co}{\mathcal O}

\author{Vítězslav Kala}
\address{Charles University, Faculty of Mathematics and Physics, Department of
	Algebra, Sokolov\-sk\' a 83, 18600 Praha~8, Czech Republic}
\email{vitezslav.kala@matfyz.cuni.cz}

\author{Daejun Kim}
\address{Department of Mathematics Education, Korea University,
	Seoul 02841, Republic of Korea}
\email{daejunkim@korea.ac.kr}

\author{Seok Hyeong Lee}
\address{Center for Quantum Structures in Modules and Spaces, Seoul National University, Seoul 08826, Republic of Korea}
\email{lshyeong@snu.ac.kr}


\thanks{The first author was supported by Czech Science Foundation (grant 21-00420M). The second author was supported by a KIAS Individual Grant (MG085501) at Korea Institute for Advanced Study. The third author was supported by the National Research Foundation of Korea (NRF) grant funded by the Korea government (MSIT) (No.2020R1A5A1016126).}

\title{Universality lifting from a general base field}

\begin{document}
	
\begin{abstract} 
Given a totally real number field $F$, we show that there are only finitely many totally real extensions of $K$ of a fixed degree that admit a universal quadratic form defined over $F$. We further obtain several explicit classification results in the case of relative quadratic extensions.
\end{abstract}

\maketitle

\section{Introduction}\label{sec:1}

Thanks to Lagrange, Ramanujan, and Bhargava, among many others, we obtained a very good understanding of \textit{universal} quadratic forms, i.e., those that represent all positive rational integers.
Results such as Bhargava--Hanke 290--Theorem \cite{BH} may even suggest that, as soon as a quadratic forms has sufficiently many variables (say, four or five), then it is actually quite easy for it to be universal.

The situation is still quite similar in rings of integers in number fields $K$ that are \textit{not} totally real. Siegel \cite{Sie1} used the circle method to show that the sum of five squares is universal over non-totally real $K$ if the discriminant of $K$ is odd (in fact, only four squares suffice, as spinor genera establish \cite{EH}); otherwise the only obstruction is dyadic. Until today, employing local methods in this indefinite situation continues to be very interesting and fruitful, e.g., \cite{HHX, XZ}.

Surprisingly, things are markedly different when one considers the \textit{totally positive} setting that we consider in this paper, i.e., totally positive definite quadratic forms representing totally positive integers $\mathcal O_K^+$ in a totally real field $K$.
In this case, already Siegel \cite{Sie1} learned that the sum of any number of squares is almost never universal, except for four squares over $\Q$ and three squares over $\Q(\sqrt 5)$ \cite{M}.

While Hsia--Kitaoka--Kneser \cite{HKK} proved the \textit{asymptotic local-global principle} that quadratic forms in at least five variables integrally represent all elements of sufficiently large norm that are represent everywhere locally, the small elements cause serious trouble. Already the aforementioned result of Siegel \cite{Sie1}  on the non-universality of sums of squares employed \textit{indecomposables}, i.e., totally positive integers $\alpha\in\co_K^+$ that are not decomposable as the sum $\alpha=\beta+\gamma$ for $\beta,\gamma\in\co_K^+$.
Recently, numerous authors 
used indecomposables and closely related tools and ideas to show that, very often, universal quadratic forms require arbitrarily many variables, first over real quadratic fields \cite{BK1,CKR, Ka1, KYZ,Ki1, Ki2, KKP},
but also in higher degrees \cite{CLSTZ,Ka3, KM,KT, KTZ, Man, Ya}. For other relevant works, often employing local-global methods, see \cite{CI,CO,Ea, EK} or the surveys \cite{Ka2,Km}.

Not only do indecomposables force universal forms to have large ranks, they also greatly restrict their possible coefficients. Specifically, the lifting problem asks the following question.

\medskip

\textbf{Lifting problem.} \textit{When does it happen that a totally positive definite quadratic form $Q$ defined over a totally real number field $F$ is universal over a totally real field $K\supset F$?}

\medskip

The first partial answer again goes back to Siegel \cite{Sie1}, as the sum of squares is defined over $\Q$ and is never universal over $K\neq\Q,\Q(\sqrt 5)$. Every diagonal form with $\Z$-coefficient is represented by the sum of squares, and so it cannot be universal either. However, over $\Q(\sqrt 5)$ there are quite a few universal forms: Already in 1928, G\" otzky \cite{Go} showed the universality of the sum of four squares. Maa{\ss} \cite{M} reduced the number of squares to three, and much later Chan--Kim--Raghavan \cite{CKR} proved that $x^2+y^2+2z^2$ is the only other universal ternary form  that is classical, i.e., whose off-diagonal coefficients are even. Deutsch then used quaternion rings to obtain the universality of $q_1=x^2+xy+y^2+z^2+zw+w^2$ \cite{D1} and $q_2=x^2+y^2+z^2+w^2+xy+xz+xw$ \cite{D2}. Despite all these results, classifying all universal forms with $\Z$-coefficients over $\Q(\sqrt 5)$ seems to be a non-trivial open problem. 
Over other fields, examples have been hard to find, and Deutsch \cite{D1} conjectured that $q_1$ is not universal over any other real quadratic field.

This conjecture was established by Kala--Yatsyna \cite{KY1} who showed that, in fact, there are no universal forms with $\Z$-coefficients over real quadratic fields different from $\Q(\sqrt 5)$. They also proved that, if $K$ has principal codifferent ideal and its degree is $3,4,5$, or $7$, then it admits no universal form with $\Z$-coefficients, unless $K=\Q(\zeta_7+\zeta_7^{-1})$, over which $q_2$ is universal (Kala--Melistas \cite{KMe} observed that the argument actually extends to some fields in degrees $\leq 43$). Gil-Mu\~ noz--Tinkov\' a \cite{GMT} considered more cubic fields, but Kim--Lee \cite{KL} then significantly strengthened these results to cover all cubic and biquadratic fields, before Kala--Yatsyna \cite{KY3} showed that in degrees $\leq 5$ there are no further fields having a universal form with $\Z$-coefficients. While the indefinite case is not the focus of our paper, there have been also very interesting recent results, e.g., by He--Hu--Xu \cite{HHX} and Xu--Zhang \cite{XZ}.

The only completely general result that is known in the totally real case concerns the \textit{weak lifting problem}: Kala--Yatsyna \cite{KY2} proved that when one fixes the base field $F$, the form $Q$, and the degree of the extension $d=[K:F]$, then there are at most finitely many such fields $K$ over which $Q$ is universal.

\medskip

Our first main theorem  significantly strengthens this result by removing the assumption that the form is fixed. As is common in the literature, throughout the paper we work in the more general language of quadratic lattices (see Section \ref{sec:2} for this and all other precise definitions), and so we already use it in the formulation of the theorem.

\begin{thm}\label{thm:intro1}
Let $F$ be a totally real number field and $d\in\n$. There are at most finitely many totally real fields $K\supseteq F$ with $[K:F]=d$ such that there is an $\of{}$-lattice $L$ such that $L\otimes\ok{}$ is universal.
\end{thm}

This will be established as Theorem \ref{cor:finiteness}. The key is to note that it is sufficient to consider the representability of individual elements (in fact, Theorem \ref{thm:equiv-existence-of-univ} says that indecomposables suffice) by lattices of bounded rank (Lemma \ref{lem:rep-by-rk<=d+1}) and that one can combine arbitrarily many representing lattices into one thanks to the existence of universality criterion sets that was recently established by Chan--Oh \cite{CO}.

\medskip

In order to obtain a more concrete understanding of the general lifting problem,  in the rest of the article we focus on the -- easiest, but already highly challenging -- case $d=2$ of quadratic extensions $K/F$.

Not surprisingly, it turns out that things are significantly easier when the base field $F$
has class number 1 for, e.g., in this case there is a relative integral basis $\co_K=\co_F[w]$. Using it, Proposition \ref{prop:quad_general_bound} gives a general finiteness criterion in terms of the relative discriminant of the extension $K/F$.

We further give a computational classification of all real quadratic fields $F$ with class number 1 and discriminant $D_F \le 200, D_F\neq 193,$ for which there is an $\of{}$-lattice with a universal lift to a quadratic extension $K/F$. The computational part consisted of a series of programs in Mathematica that were quite fast to run for small $D_F$, took approx. 7 hours for $D_F=177$, and did not finish within 1 day for the excluded case $D_F=193$, see Remark \ref{rmk:code}.

\begin{thm}\label{thm:intro2}
Table $\ref{table1}$ provides the complete list of real quadratic fields $F=\q(\sqrt{D_F})$ and totally real quadratic extensions $K/F$ such that there is an $\of{}$-lattice that is universal over $\ok{}$, among all fundamental discriminants $D_F \le 200, D_F\neq 193,$ such that $F$ has class number $1$. For all of them, $D_F\leq 56$.
\end{thm}

This will be proved as Theorem \ref{thm:table}. We first use Proposition \ref{prop:quad_general_bound} to get a finite list of possible fields $K$, and then for each of them we check the $F$-representability (i.e., representability by an $\co_F$-lattice) of suitable (indecomposable) elements. This restricts us to the candidate fields from Table \ref{table1}. For each of these, we 
prove the existence of an $\co_F$-lattice whose lift to $K$ is universal by using Theorem \ref{thm:equiv-existence-of-univ}(3). To do this, we find all the indecomposables in $K$ using the geometry of numbers method based on Shintani's unit theorem, first developed by Kala--Tinkov\' a \cite{KT}, and then check their $F$-representability using Lemma \ref{lem:solve-pqr}.

As fields $F$ with larger discriminant $D_F$ seem to be much less likely to admit universality lifting, we in fact formulate
Conjecture \ref{conj:real quad lifts} that the list in Table \ref{table1} is complete.

While the class number 1 assumption is restrictive, it is not overly so. Out of 30 fundamental discriminants $<100$, only 4 (viz., $40,60,65,85$) give class number greater than 1; between $100$ and $200$, it is 10 out of 30. Also for Conjecture \ref{conj:real quad lifts}, the very hard Class Number One Problem (and, e.g., Cohen--Lenstra heuristics) predicts that there are 
infinitely many real quadratic fields with class number 1.

\medskip

Finally, in order to go beyond the restriction placed by the assumption that $F$ has class number 1, we consider extensions by square roots of positive integers.

\begin{thm}\label{thm:intro3}
Let $F$ be a totally real number field and $K=F(\sqrt e)$ where $e>0$ is a square-free integer such that the discriminants of $F$ and $\Q(\sqrt e)$ are coprime.

If $e\neq 5$, then there is no $\of{}$-lattice $L$ such that $L\otimes\ok{}$ is universal over $\ok{}$.

If $F$ is real quadratic and there is an $\of{}$-lattice $L$ such that $L\otimes \co_{F(\sqrt 5)}$ is universal over $\co_{F(\sqrt 5)}$, then $D_F\le 4076$. 
\end{thm}

This result is contained in Theorems \ref{prop:lifting-fail-sqrt(d)-extn} and \ref{thm:extn-by-sqrt5}.
Note that it provides extra evidence for Conjecture \ref{conj:real quad lifts} in the case when $K$ is a real biquadratic field.

\medskip

To conclude and summarize the Introduction, our results strongly back up the expectation that universality lifting is quite rare. Significantly, we now have data supporting this in cases of more general base fields than just the rationals $\Q$. One may be tempted to hypothesize that the finiteness result of Theorem \ref{thm:intro1} holds even without fixing the relative degree $d$, and then attempt to establish more classification results akin to our Theorems \ref{thm:intro2} and \ref{thm:intro3}. We anticipate that the further development of these ideas will lead to an exciting research program requiring many new ideas.

\section*{Acknowledgments}

We thank Jakub Kr\' asensk\' y for interesting and helpful discussions about the paper.

\section{Preliminaries}\label{sec:2}

\subsection{Notations and terminologies} 
Let $K/F$ be number fields of a degree $d=[K:F]$ and let $\sigma_1,\ldots,\sigma_d$ be embeddings from $K$ into $\mathbb{C}$ which fix $F$. The (relative) {\em norm} map and the (relative) {\em trace} map $N_{K/F}:K\ra F$ and $\tr{K/F}:K\ra F$ are defined by $N_{K/F}(\alpha)=\sigma_1(\alpha)\cdots \sigma_d(\alpha)$ and $\tr_{K/F}(\alpha)=\sigma_1(\alpha)+\cdots + \sigma_n(\alpha)$ for $\alpha\in K$, respectively. In case when the base field $F=\q$, we simply write the norm and the trace map as $N$ and $\tr$. The {\it relative discriminant} $\Delta_{K/F}$ is defined to be the ideal in $F$ generated by $\{\det(\tr_{K/F}(a_i a_j))_{i,j=1,\ldots, d} : a_1, \ldots, a_d \in \ok{}\}$. Note that when $\ok{}$ is free $\of{}$-module with a basis $b_1, \ldots, b_d$, then $\Delta_{K/F}$ is the principal ideal generated by $\det(\tr_{K/F}(b_i b_j))$.

Throughout this paper, $F$ will denote a totally real number field of degree $n=[F:\q]$ over $\q$, and $K$ will denote a totally real number field containing $F$ such that $[K:F]=d$. Let $\rho_1=\mathrm{id}, \ldots , \rho_n :F\ra \mathbb{R}$ be the real embeddings of $F$. We say $\alpha\in F$ is {\em totally positive} if $\rho_i(\alpha)>0$ for all $1\le i\le n$, and write $\alpha \succ 0$. We denote by $\of{+}$ the set of all totally positive elements in $\of{}$ and denote by $\of{\times}$ the set of all units of $\of{}$. Moreover, we write $\of{\times,+}=\of{\times}\cap \of{+}$. An element $\alpha\in\of{+}$ is called {\em indecomposable} if it cannot be written as $\alpha=\beta+\gamma$ with $\beta,\gamma\in\of{+}$.
	
Now we introduce the geometric language of quadratic spaces and lattices. A {\em quadratic space} over $F$ is a vector space $V$ over $\q$ equipped with a non-degenerate symmetric bilinear form $B:V\times V \ra \q$. We say $V$ is {\em totally positive definite} if the associated quadratic form $Q(v)=B(v,v)$ is totally positive definite, that is, $Q(v)>0$ for all $v\in V\setminus\{0\}$. Throughout the paper, we assume that $V$ is totally positive definite, unless stated otherwise.

A {\em quadratic $\of{}$-lattice} $L$ on $V$ is a finitely generated $\of{}$-module such that $FL=V$. We write $(L,Q)$ to indicate that the $\of{}$-lattice $L$ is associated with the quadratic map $Q$. For $\alpha\in\of{+}$, we say $L$ {\em represents} $\alpha$ if $Q(v)=\alpha$ for some $v\in L$. We say an $\of{}$-lattice is {\em universal} over $\of{}$ if it represents every element in $\of{+}$. For a field $K$ containing $F$, we may consider $L$ as an $\ok{}$-lattice, namely, we may consider the $\ok{}$-lattice $L\otimes_{\of{}}\ok{}$. Hence we say an $\of{}$-lattice $L$ is {\em universal} over $\ok{}$ if $L\otimes\ok{}$ is universal as an $\ok{}$-lattice. We may identify a free lattice with a quadratic form, namely, for $L=\of{}v_1 + \cdots + \of{} v_r$ we consider
\[
	Q_L(x_1,x_2,\ldots,x_r)=Q(x_1v_1+\cdots+x_rv_r)=\sum_{1\le i\le j\le r} a_{ij} x_ix_j \text{ with } a_{ij}\in\of{}.
\]

If the ring $\of{}$ is a principal ideal domain, then an $\of{}$-lattice $L$ has an integral basis, meaning that $L=\of{} v_1 + \cdots + \of{} v_r$ for a basis $v_1,\ldots ,v_r$ of $V$. In the case when $\of{}$ is not be a principal ideal domain, $L$ may not have an integral basis, but it can be generated by $\dim(V)+1$ elements as an $\of{}$-module (see 22:5 and 81:5 of \cite{OM2}). If we are given quadratic spaces $V_i$ over $F$ and $\of{}$-lattices $L_i$ on $V_i$ for $1\le i \le t$, the {\em orthogonal sum} of $L_i$ is the $\of{}$-lattice $L=L_1\perp \cdots \perp L_t$ on $V=V_1\perp \cdots \perp V_t$.

Any unexplained notation and terminology on quadratic lattices can be found in \cite{OM2}.

\subsection{Real quadratic number fields}\label{sec:real-quad-field}

Let $F = \q(\sqrt{d})$ be a real quadratic number field with $d>0$ a square-free integer. The ring $\of{}$ of algebraic integers of $F$ is described as
\begin{equation}\label{eqn:def-O_F}
\oo_F = \z+\z \tau, \text{ where } \tau:=\begin{cases} \frac{1+\sqrt{d}}{2} & \text{if } d \equiv 1 \Mod{4},\\
\sqrt{d} &\text{if } d \equiv 2,3 \Mod{4},\end{cases}
\end{equation}
and the discriminant $D_F:=\mathrm{Disc}(F)$ of $F$ is given as
\begin{equation}\label{eqn:def-D_F}
D_F = \begin{cases} d & \text{if } d \equiv 1 \Mod{4}, \\ 4d & \text{if } d \equiv 2,3 \Mod{4}.\end{cases}
\end{equation}

An integer $D$ is called a {\em fundamental discriminant} if it is the discriminant of a quadratic number field. Note from \eqref{eqn:def-D_F} that $D>1$ is a fundamental discriminant if and only if it is not divisible by any square of any odd prime and satisfies $D \equiv 1 \Mod{4}$ or $D \equiv 8,12 \Mod{16}$.

\section{General properties for lifting problem}\label{sec:3}

In this section, we discuss some properties of general lifting problem, asking for which totally real fields $F\subseteq K$ there exists an $\of{}$-lattice $L$ such that $L\otimes\ok{}$ is universal over $\ok{}$. Let us first introduce the following definition.

\begin{defn} Let $K,F$ be totally real fields such that $F \subseteq K$. We say an element $\alpha \in \ok{+}$ is {\it $F$-representable} if there exists a totally positive definite $\of{}$-lattice $L$ such that $L \otimes_{\of{}} \ok{}$ represents $\alpha$. 
\end{defn}

\begin{lem}\label{lem:rep-by-rk<=d+1}
	Let $K,F$ be totally real fields such that $F \subseteq K$ and $[K:F]=d$.
	If $\alpha\in\ok{+}$ is $F$-representable, then there exists an $\of{}$-lattice $L(\alpha)$ of rank at most $d+1$ such that $L(\alpha)\otimes\ok{}$ represents $\alpha$. In particular, if $\ok{}$ is generated by $g$ elements as an $\of{}$-module, say $\ok{}=\of{}w_1+\cdots+\of{}w_g$, then there is a totally positive semidefinite quadratic form $Q$ such that $Q(w_1,\ldots,w_g)=\alpha$.
\end{lem}
\begin{proof}
	Assume that there is an $\of{}$-lattice $\ell$ of rank $r$ such that $\ell\otimes \ok{}$ represents $\alpha$. As $\ell$ and $\ok{}$ are finitely generated $\of{}$-module, let us fix generators of them as follows:
	\[
		\ell=\of{} v_1 + \cdots + \of{} v_r + \of{} v_{r+1} \quad \text{and} \quad \ok{}=\of{}w_1+\cdots+\of{}w_{g},
	\]
	where $g$ should be no less than $d$ and $g$ can be taken to satisfy $g\le d+1$.
	As $\alpha$ is represented by $\ell\otimes\ok{}$, there are $\beta_1,\ldots,\beta_{r+1}\in\ok{}$ such that $Q(\beta_1v_1+\cdots+\beta_{r+1}v_{r+1})=\alpha$. 
	
	Writing $\beta_i=\sum_{j=1}^{g} a_{ij}w_j$ with $a_{ij}\in\ok{}$ and putting $A=(a_{ij})_{(r+1)\times g}$, $M_\ell=(B(v_i,v_j))_{(r+1)\times(r+1)}$, $\bm{\beta}=(\beta_1,\ldots,\beta_{r+1})\in\ok{r+1}$, and $\bm{w}=(w_1,\ldots,w_{g})\in\ok{g}$, we have $\bm{\beta}=A\bm{w}$ and
	\begin{equation}\label{eqn:Q0repalpha}
		\alpha = Q(\beta_1v_1+\cdots+\beta_{r+1}v_{r+1})=\bm{\beta}^t M_\ell \bm{\beta} = \bm{w}^t(A^tM_\ell A)\bm{w}.
	\end{equation}
	We may associate a free $\of{}$-lattice $(L=\of{g},Q_0)$ which is totally positive semidefinite such that the Gram matrix of $Q_0$ is $A^tM_\ell A$. This $Q_0$ with \eqref{eqn:Q0repalpha} proves the in particular part of the lemma.
	
	We may assume that we have taken $g=d+1$. If $L$ is totally positive definite, then $L(\alpha):=L$ is the desired $\of{}$-lattice.
	In the case when $L$ is not totally positive definite, let us consider the {\it radical}  $\text{rad}(L):=\{v \in L : B(v,w)=0 \text{ for all } w\in L\}$ of $L$. We have the radical splitting $L=L'\perp \text{rad}(L)$, where $L'$ is a {\it regular} sublattice of $L$, meaning that $\text{rad}(L')=0$ (see \cite[p. 226]{OM2}). Note that $(L',Q=Q_0|_{L'})$ is of rank $\le d+1$ and $L'$ is totally positive definite: if $Q(v)=0$ for some $v\in L'$, then the Cauchy-Schwarz inequality $Q(v)Q(w)-B(v,w)^2\succeq 0$ implies that $B(v,w)=0$ for all $w\in L'$. Hence $v\in \text{rad}(L')=0$, which yields $v=0$. Furthermore, one may easily show that $L'$ represents $L$. Therefore, we may take $L(\alpha):=L'$. This proves the lemma.
\end{proof}

\begin{thm}\label{thm:equiv-existence-of-univ}
	Let $K,F$ be totally real fields such that $F \subseteq K$ and $[K:F]=d$. The following are equivalent.
	\begin{enumerate}[label={{\rm(\arabic*)}}]
		\item There exists an $\of{}$-lattice $L$ such that $L\otimes \ok{}$ is universal.
		\item Every element in $\ok{+}$ is $F$-representable.
		\item Every indecomposable element in $\ok{+}$ is $F$-representable.
	\end{enumerate} 
\end{thm}
\begin{proof}
	The implications $(1)\Rightarrow(2)$ and $(2)\Rightarrow(3)$ are trivial from the definition. To show the implication $(3)\Rightarrow (1)$, we use \cite[Proposition 7.1]{KT} which says that the diagonal quadratic form 
	$$\sum_{\sigma\in\mathcal S}\sigma\left(x^2_{1,\sigma}+x^2_{2,\sigma}+\dots+x^2_{s,\sigma} \right)$$
	is universal over $\co_K$, where $s$ is the Pythagoras number of $\co_K$ (i.e., the smallest positive integer such that every sum of squares equals the sum of $s$ squares) and $\mathcal S$ denotes a finite set of representatives of classes of indecomposables in $\co$ up to multiplication by squares of units $\co^{\times 2}$.
	From the assumption, there are $\of{}$-lattices $L(\sigma)$ such that $L(\sigma)\otimes\ok{}$ represents $\sigma\in\mathcal S$.
 
 Therefore the $\of{}$-lattice 
 $$\perp_{\sigma\in\mathcal S}L(\sigma)^{\perp s}$$
 is universal over $\ok{}$.
\end{proof}

\begin{thm}\label{cor:finiteness}
Let $F$ be a totally real number field and $d\in\n$. There are at most finitely many totally real fields $K\supseteq F$ with $[K:F]=d$ such that there is an $\of{}$-lattice $L$ such that $L\otimes\ok{}$ is universal.
\end{thm}

\begin{proof}
Let us start by considering a field $K\supseteq F$ with $[K:F]=d$ such that there is an $\of{}$-lattice $L$ such that $L\otimes\ok{}$ is universal. By Lemma \ref{lem:rep-by-rk<=d+1}, for each element $\alpha\in\ok{+}$, there is an $\of{}$-lattice $L(\alpha)$ of rank at most $d+1$ such that $L(\alpha)\otimes\ok{}$ represents $\alpha$. 

Now we use a theorem of Chan--Oh \cite[Theorem 5.7]{CO} that for an infinite set $\mathscr{S}$ of $\of{}$-lattices of fixed rank, there exists a finite subset $\mathscr{S}_0$ of $\mathscr{S}$ such that if an $\of{}$-lattice represents all lattices in $\mathscr{S}_0$, then it represents all lattices in $\mathscr{S}$. Thus, there even exists an $\of{}$-lattice $L_\mathscr{S}$ that represents all $\of{}$-lattices in $\mathscr{S}$: for example, $L_\mathscr{S}$ can be taken as $\perp_{\ell\in\mathscr{S}_0}\ell$. 
In particular, for each $k\in\n$ we can take an $\of{}$-lattice $L_k$ that represents all $\of{}$-lattices of rank $k$. 

Let  $\mathcal{L}=L_1\perp\cdots\perp L_{d+1}$. Then $\mathcal{L}$ represents $L(\alpha)$ for all $\alpha\in\ok{+}$, and hence $\mathcal{L}\otimes\ok{}$ represents all $\alpha\in\ok{+}$, i.e., it is universal over $K$.

We have thus established that if $K$ admits a universal $\of{}$-lattice, then in fact $\mathcal{L}\otimes\ok{}$ is universal over $K$.
But this is a situation when we can apply \cite[Theorem 2]{KY2} according to which there are only finitely many such fields $K$.
\end{proof}

\section{Lifting problems for quadratic extensions over PIDs}\label{sec:4}

In this section, we discuss lifting problems for quadratic extensions $K/F$ of totally real number fields where the base field $F$ has class number 1 (i.e., $\of{}$ is a principal ideal domain). In Section \ref{subsec-whenOFisPID}, we then further specialize to the case when $F$ is a real quadratic field.

First, we give a certain finiteness bound on relative discriminant of $K/F$ for those $K$ having an $\of{}$-lattice such that $L\otimes\ok{}$ is universal. 

\begin{prop}\label{prop:quad_general_bound}
	Let $K/F$ be a quadratic extension of totally real number fields, where $F$ has class number $1$, and let $\Delta_{K/F}$ be the relative discriminant ideal. Fix a set of representatives $U_F$ of the additive group $\of{}/2\of{}$. Then: 
	\begin{enumerate}[label={\rm{(\arabic*)}}, leftmargin=*]
        \item There exists an element $\Delta \in \of{+}$ that is a generator of $\Delta_{K/F}$ and $K=F(\sqrt{\Delta})$. Such  $\Delta\in\of{+}$ is determined uniquely up to multiplication by squares of units in $\of{}$.

		\item For a given $\Delta\in\of{+}$ satisfying $\mathrm{(1)}$, there is a unique element $w_\Delta\in\ok{}$ such that

		\begin{equation}\label{eqn:condw_Delta}
		\ok{}=\of{}1+\of{}w_\Delta, \ t\in U_F, \ \Delta=t^2-4n, \text{ and } w_\Delta=\frac{t+\sqrt{\Delta}}{2},
		\end{equation}
		where $t:=\tr_{K/F}(w_\Delta)$ and $n:=\mathrm{N}_{K/F}(w_\Delta)$.
		\item Assume that a generator $\Delta\in\of{+}$ of $\Delta_{K/F}$ in $\mathrm {(1)}$ satisfies the following: 
		\begin{equation}\label{eqn:condDelta}
			\Delta \succeq u^2 \text{ and } \Delta \succeq (2-u)^2 \text{ for all } u \in U_F; \text{ and additionally } \Delta \succeq 9 \text{ if }\Delta \equiv 1 \Mod{ 4\of{}}.
		\end{equation}
		Then there exists $m \in \of{}$ such that $m+w_\Delta$ is totally positive but not $F$-representable. Thus for $K=F(\sqrt{\Delta})$, there is no $\of{}$-lattice $L$ such that $L\otimes\ok{}$ is universal.
	\end{enumerate}
\end{prop}

\begin{proof}
    (1) As $\of{}$ is principal and $1\in\of{}$ is primitive, one can consider an $\of{}$-basis $\{1,v\}$ of $\ok{}$, namely, $\ok{}=\of{}1+\of{}v$. For such a basis, define $\Delta_v:=\det\left(\begin{smallmatrix}
		\tr_{K/F}(1) & 	\tr_{K/F}(v)\\ 	\tr_{K/F}(v) & 	\tr_{K/F}(v^2)
	\end{smallmatrix}\right)$. One may observe that $\Delta_v=\tr_{K/F}(v)^2-4\mathrm{N}_{K/F}(v)$, and $\Delta_v$ is a generator of the relative discriminant $\Delta_{K/F}$. As $v^2-\tr_{K/F}(v)v+\mathrm{N}_{K/F}(v)=0$, we have $v=\frac{\tr_{K/F}(v)\pm\sqrt{\Delta_v}}{2}$. Thus $K=F(v)=F(\sqrt{\Delta})$, and $\Delta\in\of{+}$ since $K$ is a totally real number field.

    To show the uniqueness, assume that $F(\Delta_1)=F(\Delta_2)$ for some generators $\Delta_1,\Delta_2\in\of{+}$ of $\Delta_{K/F}$. Then $\Delta_2=\Delta_1 f^2$ for some $f\in F$, and since $\Delta_1,\Delta_2$ are generators of $\Delta_{K/F}$, we have $f\in\of{\times}$.
	
	(2) We first show the existence. From (1) and under the same notation given in the proof of (1), we have $\Delta=\Delta_v \epsilon^2$ for some $\epsilon\in\of{\times}$. Letting $w=v\epsilon^{-1}$, we have
    \[
        w = v \epsilon^{-1} = \frac {\epsilon^{-1} \tr_{K/F}(v)  + \sqrt{ \epsilon^{-2} \Delta_v}}{2} = \frac {\tr_{K/F}(w)  + \sqrt{ \Delta}}{2}
    \]
    Since $\tr_{K/F}(w-c)=\tr_{K/F}(w)-2c$ for any $c\in\of{}$, there exists a unique $c\in\of{}$ such that $\tr_{K/F}(w-c)\in U_F$. If we put $w_\Delta=w-c$ for such $c\in\of{}$, then one may easily check that $w_\Delta$ (or its conjugate) satisfies the conditions in \eqref{eqn:condw_Delta}.

	To show the uniqueness, consider two $w, w'\in\ok{}$ satisfying \eqref{eqn:condw_Delta} instead of $w_\Delta$. From $\ok{}=\of{}1+\of{}w=\of{}1+\of{}w'$, we may write $w'=a+bw$ for some $a,b\in\of{}$. Since $\Delta=\Delta_{w'}=b^2\Delta_{w}=b^2\Delta$, we have $b=\pm 1$. Assume that $b=-1$, that is, $w'=a-w$. Then noting that $\tr_{K/F}(w')=2a-\tr_{K/F}(w)$, we have
	\[
		a-\frac{\tr_{K/F}(w)+\sqrt{\Delta}}{2} = a-w = w' = \frac{\tr_{K/F}(w')+\sqrt{\Delta}}{2}=\frac{2a-\tr_{K/F}(w)+\sqrt{\Delta}}{2},
	\]
	which is impossible since this yields $\sqrt{\Delta}=-\sqrt{\Delta}$. Therefore, we have $b=1$, that is, $w'=a+w$, and the condition $\tr_{K/F}(w), \tr_{K/F}(w')\in U_F$ implies that $a=0$, equivalently, $w=w'$.
		
	(4) Note that if we find $m\in\of{}$ such that $m+w_\Delta\in\ok{+}$ which is not $F$-representable, Theorem \ref{thm:equiv-existence-of-univ} implies that there is no $\of{}$-lattice $L$ such that $L\otimes\ok{}$ is universal for $K=F(\sqrt{\Delta})$.
	
	Let $m$ be any element of $\of{}$ such that $m + w_\Delta\in\ok{+}$. Observe that if $m + w_\Delta$ is $F$-representable, then Lemma \ref{lem:rep-by-rk<=d+1} implies that there is a totally positive semidefinite $\of{}$-form $Q(x,y)=px^2+qxy+ry^2$ in two variables such that 
	\[
	m+ w_\Delta = Q(1, w_\Delta) = p + q w_\Delta + r w_\Delta^2.
	\]
	Here the coefficients $p,q,r \in \of{}$ satisfies $p,r, 4pr-q^2 \succeq 0$. Noting that $w_\Delta^2=tw_\Delta-n$ where $t:=\tr_{K/F}(w_\Delta)$ and $n:=\mathrm{N}_{K/F}(w_\Delta)$, we have 
	\[
		p+qw_\Delta+rw_\Delta^2=(p - nr) + (q + tr) w_\Delta
	\]	
	Since $\ok{}=\of{}1+\of{}w_\Delta$, comparing coefficients of both sides together with $n=\frac{t^2-\Delta}{4}$ yields
	\begin{equation}\label{eqn:F-repconditions}
		p + \frac{\Delta - t^2}{4}r = m \quad \text{and} \quad q + tr = 1.
	\end{equation}
	If there is no solution $p,q,r\in\of{}$ of the above equation which satisfies $p,r, 4pr-q^2 \succeq 0$, then we may conclude that $m + w_\Delta$ is not $F$-representable.
	
	Take $m$ as $m_1 := \frac{\Delta - t^2}{4}\in\of{}$. Since $\Delta \succeq t^2$ and $\Delta \succeq (2-t)^2$ from the assumption, we have $m+w_\Delta\in\ok{+}$. Indeed,
	\[
	m_1 + w_\Delta = \frac{\Delta - t^2}{4} + \frac{t+\sqrt{\Delta}}{2} = \frac{\sqrt{\Delta} + t}{2} \left(\frac{\sqrt{\Delta}-t}{2}+1\right),
	\]
	and this number becomes totally positive if $\rho(\Delta) \ge \rho(t)^2$ and $\rho(\Delta) \ge \rho(2-t)^2$ for all embeddings $\rho$ of $F$. Moreover, we have $m_1 \succ 0$. Note that the first equation of \eqref{eqn:F-repconditions} is $p+ m_1r = m_1$. If $p,r\neq 0$, then considering the norm of both sides we have
	\[
		\mathrm{N}_{F/\q}(m_1)^{\frac{1}{d}}=\mathrm{N}_{F/\q}(p+m_1r)^{\frac{1}{d}}\ge \mathrm{N}_{F/\q}(p)^{\frac{1}{d}}+\mathrm{N}_{F/\q}(m_1r)^{\frac{1}{d}}> \mathrm{N}_{F/\q}(m_1r)^{\frac{1}{d}}\ge\mathrm{N}_{F/\q}(m_1)^{\frac{1}{d}},
	\]
	where $d=[F:\q]$, which is a contradiction. Thus either $p=0$ or $r=0$. Then we have $q=0$ since $4pr-q^2 \succeq 0$. On the other hand, as the second equation of \eqref{eqn:F-repconditions} is $q+tr = 1$ we have $r = 1/t \neq 0$. Thus $p=0$ and it follows from $p + m_1 r = m_1$ that $r=1$, hence $t=1$.
	
	Now suppose $t=1$. Then we can write $\Delta = 1+4n'$ where $n'=\frac{\Delta-1}{4}\in\of{+}$. Let $m_2 = n'-1$ and note that
	\[
	m_2 + w_\Delta = n'-1 + \frac{1+\sqrt{1+4n'}}{2}.
	\]
	Since we are assuming $\Delta\succeq 9$, we have $\rho(n') \ge 2$ for all embeddings $\rho$ of $F$. Hence one may easily show that
	\[
	\rho(n')-1 + \frac{1 - \sqrt{1 + 4 \rho(n')}}{2} \ge 0.
	\]
	Therefore we have $m_2 + w_\Delta \in\ok{+}$. 
	Note further that $\mathrm{N}_{F/\q}(n')>\mathrm{N}_{F/\q}(n'-1)$ since $\rho(n')\ge\rho(n'-1)\ge1$ for all embeddings $\rho$ of $F$.
	Taking $m$ as $m_2$ in \eqref{eqn:F-repconditions}, the first equation $p + n' r = n'-1$ implies that $r=0$ since otherwise,
	\[
	\mathrm{N}_{F/\q}(n'-1)^{\frac{1}{d}}=\mathrm{N}_{F/\q}(p+rn')^{\frac{1}{d}}\ge \mathrm{N}_{F/\q}(p)^{\frac{1}{d}}+\mathrm{N}_{F/\q}(rn')^{\frac{1}{d}}\ge \mathrm{N}_{F/\q}(rn')^{\frac{1}{d}}\ge\mathrm{N}_{F/\q}(n')^{\frac{1}{d}},
	\]
	which is a contradiction. Then the second equation in \eqref{eqn:F-repconditions} yields $q=0$, but this is impossible since the condition $4pr-q^2=-1\succeq 0$ does not hold.
\end{proof}

For further use, 
let us prepare a general condition for $F$-representability.

\begin{lem}\label{lem:solve-pqr}
	Let $F$ be a totally real number field. Let $K=F(\sqrt{\Delta})$ be a totally real quadratic extension of $F$ with $\Delta\in\of{+}$ such that $\ok{}=\of{}1+\of{}w_\Delta$ with $w_\Delta=\frac{t+\sqrt{\Delta}}{2}$, where this condition holds for every $F$ of class number $1$ by Proposition $\ref{prop:quad_general_bound}$. For $u,v\in\of{}$, the following are equivalent.

	\begin{enumerate}[label={\rm{(\arabic*)}}]
		\item $\alpha = u+vw_\Delta\in\ok{+}$ is $F$-representable.
		\item There exist $p,q,r\in\of{}$ such that $p,r,4pr-q^2\succeq 0$ and 
	\begin{equation}\label{eq:pqr}
p + q w_{\Delta} + r w_{\Delta}^2 = u + v w_{\Delta}.		    
		\end{equation} 
	\end{enumerate}
 
Moreover, for every $p,q,r$, \eqref{eq:pqr} is equivalent to 
		\[
		p + \frac{\Delta - t^2}{4}r = u \quad \text{and} \quad q + tr = v.
		\]

	Finally, if $(1)$ and $(2)$  hold, then we have
 \[
 2u + vt \succeq 0 \quad \text{and} \quad r \preceq \frac{4u+2vt}{\Delta}.
 \]
\end{lem}
\begin{proof}
	By Lemma \ref{lem:rep-by-rk<=d+1}, $\alpha$ is $F$-representable if and only if there is a positive semidefinite quadratic form $Q=px^2+qxy+ry^2$ such that $Q(1,w_\Delta)=\alpha=u+vw_\Delta$. Noting that 
	\[
	Q(1,w_\Delta)=p+qw_\Delta + rw_\Delta^2 = p+qw_\Delta + r\left(\frac{\Delta - t^2}{4}+tw_\Delta\right),
	\] 
	it is equivalent to that there is a solution $p,q,r\in\of{}$ satisfying $p,r,4pr-q^2\succeq 0$ to the equation
	\begin{equation*}
		p + \frac{\Delta - t^2}{4}r = u \quad \text{and} \quad q + tr = v,
	\end{equation*}
	which proves the equivalence of $(1)$ and $(2)$ and the `Moreover' claim.
	
	To prove the `Finally' part, note that $2\alpha=2u+2vw_\Delta=(2u+vt)+v\sqrt{\Delta}\in\ok{+}$ implies that $\rho(2u+vt)\pm \rho(v)\sqrt{\Delta}\ge 0$ for any embedding $\rho:F\ra \mathbb{R}$. Hence we have $\rho(2u+vt)\ge0$. 
	Now putting $p=u-\frac{\Delta-t^2}{4}r$ and $q=v-tr$ into $4pr-q^2\succeq0$, we have $\Delta r^2-2(2u+vt)r+v^2 \preceq 0$. Thus for any embedding $\rho : F \ra \mathbb{R}$ we have
 \[
 \rho(\Delta r^2-2(2u+vt)r+v^2) =  \rho(\Delta) \rho(r)^2 - 2\rho(2u+vt) \rho(r) + \rho(v)^2 \le 0.
 \]
 As we have $\rho(\Delta) >0$, solving this quadratic inequality in $\rho(r)$ gives
	\begin{equation}\label{eqn:rbound}
		\frac{\rho(2u+vt)-\sqrt{\rho(2u+vt)^2-\rho(\Delta) \rho(v)^2}}{\rho(\Delta)} \le \rho(r) \le \frac{\rho(2u+vt)+\sqrt{\rho(2u+vt)^2-\rho(\Delta) \rho(v)^2}}{\rho(\Delta)}.
	\end{equation}
As $\rho(\Delta)>0$ and $\rho(2u+vt) \ge 0$, we have
\[
\rho(r) \le \frac{\rho(2u+vt)+\sqrt{\rho(2u+vt)^2-\rho(\Delta) \rho(v)^2}}{\rho(\Delta)} \le \frac{\rho(2u+vt)+\sqrt{\rho(2u+vt)^2}}{\rho(\Delta)} = \frac{2\rho(2u+vt)}{\rho(\Delta)}.
\]
Hence we obtain $r \preceq \dfrac{4u+2vt}{\Delta}$, completing the proof of the lemma.
\end{proof}

\section{When $F$ is a real quadratic field of class number $1$}\label{subsec-whenOFisPID}
In this section, we explore the lifting problem when the base field $F$ is a real quadratic field of class number $1$. Let $F = \q(\sqrt{D})$ be a real quadratic field of class number $1$, where $D>0$ is a fundamental discriminant. Throughout this section, let us denote two embeddings of $F$ into $\mathbb{R}$ by $\rho_1$ and $\rho_2$.

Let $K/F$ be a quadratic extension of totally real number fields which admits an $\of{}$-lattice $L$ such that $L\otimes\ok{}$ is universal. We will classify all such fields $K=F(\sqrt{\Delta})$ with $\Delta\in\of{+}$ (see Proposition \ref{prop:quad_general_bound}). In addition to Proposition \ref{prop:quad_general_bound}, we first give another way of bounding $\Delta$ when $F$ is a real quadratic field.

\begin{lem} \label{lem:quadRounding}
	Let $F$ be a real quadratic number field with discriminant $D=D_F$ and with two real embeddings $\rho_1, \rho_2 : F \rightarrow \r$. For any $x, y \in \r$, there exists $\alpha \in \of{}$ such that
	\[
	x \le \rho_1(\alpha) < x + l_F  \quad \text{and} \quad y \le \rho_2(\alpha) < y + l_F
	\]
	where $l_F = \frac{1}{2} \sqrt{D} + 1$.
\end{lem}

\begin{proof}
Recall that $\of{}=\Z 1 + \Z \tau$ where $\tau \in \of{}$ is defined as in \eqref{eqn:def-O_F}. Pick integers $k$ such that $x-y \in \left(-\frac{\sqrt{D}}{2},\frac{\sqrt{D}}{2} \right] + k \sqrt{D}$, and consider $(x', y') = (x - \rho_1 (k \tau), y - \rho_2 (k \tau)) $. As $\rho_1(\tau) - \rho_2(\tau) = \sqrt{D}$, we have
\[
x'-y' = (x-y) - k(\rho_1 (\tau) - \rho_2(\tau))  = x- y- k \sqrt{D}  \in \left(-\frac{\sqrt{D}}{2},\frac{\sqrt{D}}{2} \right].
\]
Pick an integer $m$ such that $x+y \in \left(-\frac{\sqrt{D}}{2}-2,-\frac{\sqrt{D}}{2} \right] + 2m$, and consider $(x'',y'') = (x - m, y - m)$.
\[
x'' - y'' = x' - y'  \in  \left(-\frac{\sqrt{D}}{2},\frac{\sqrt{D}}{2} \right] \quad \text{and}\quad x'' + y'' = x' + y' - 2m \in \left(-\frac{\sqrt{D}}{2}-2,-\frac{\sqrt{D}}{2} \right].
\]
By adding and subtracting two equations, we obtain
\[
2x'', 2y'' \in (-\sqrt{D}-2, 0] \quad \Rightarrow \quad x'' \le 0 < x'' + l_F \text{ and } y'' \le 0 < y'' + l_F.
\]
Noting that $(x'', y'') = (x' - m, y' - m) = (x - \rho_1(m + k \tau), y - \rho_2(m+ k \tau))$, we have
\[
x \le \rho_1(m + k \tau)  < x + l_F \quad \text{and} \quad y \le \rho_2(m + k \tau) < y + l_F.
\]
Thus $\alpha=m+k\tau\in\of{}$ satisfies the conditions of the lemma.
\end{proof}

\begin{prop}\label{prop:additional-bound-Delta}
	Let $K=F(\sqrt{\Delta})$ be a totally real quadratic extension of a real quadratic field $F$ of class number $1$. Let $\Delta \in \of{+}$ and $w_\Delta  = \frac{t + \sqrt{\Delta}}{2}\in\ok{}$ be constructed as in Proposition $\ref{prop:quad_general_bound}$. If
	\begin{equation}\label{eqn:add-condDelta}
	\rho_1(\Delta) \rho_2(\Delta) \ge (2 \sqrt{\rho_1(\Delta)} + 4 l_F) (2 \sqrt{\rho_2(\Delta)} + 4l_F),
	\end{equation}
	then there exists $m \in \of{}$ such that $m + w_\Delta\in\ok{+}$ and $m + w_\Delta$ is not $F$-representable.
\end{prop}

\begin{proof}
	By the Lemma \ref{lem:quadRounding}, there exists an $m \in \of{}$ such that for any $i=1,2$
	\begin{equation} \label{eqn:prop-add-bd-Delta}
		\frac{-\rho_i(t) + \sqrt{\rho_i(\Delta)} } {2} \le \rho_i(m) < \frac{-\rho_i(t) + \sqrt{\rho_i(\Delta)} } {2} + l_F.
	\end{equation}
	Noting that all conjugates of $m + w_\Delta$ in $K$ are $\rho_i(m) + \frac{\rho_i(t) \pm \sqrt{\rho_i(\Delta)} } {2}$ with $i=1,2$, we have $m + w_\Delta\in\ok{+}$.
	
	Now we will show that $m + w_\Delta$ is not $F$-representable. Suppose on the contrary that $m+w_{\Delta}$ is $F$-representable. By Lemma \ref{lem:solve-pqr}, there exists $p,q,r \in \of{}$ such that 
 \begin{equation}\label{eqn:Deltar-cond}
    p,r, 4pr - q^2 \succeq 0, \quad m + w_{\Delta} = p + q w_{\Delta} + r w_{\Delta}^2, \quad \text{and} \quad \Delta r \preceq 4m + 2t.
\end{equation}
    If $r=0$, then from $4pr - q^2 =-q^2 \succeq 0$ it follows that $q=0$ and $m + w_{\Delta}=p$, which is contradiction. Thus $r \succ 0$ and $N(r) =\rho_1(r) \rho_2(r) \ge 1$. Applying \eqref{eqn:prop-add-bd-Delta} and \eqref{eqn:Deltar-cond}, we may verify that
 \begin{multline*}
     \rho_1 (\Delta) \rho_2(\Delta)  \le \rho_1(\Delta)\rho_1(r) \rho_2(\Delta) \rho_2(r) \le \rho_1 (4m+2t) \rho_2(4m+2t) \\
     = ( 4 \rho_1(m) + 2 \rho_1(t) ) ( 4 \rho_2(m) + 2 \rho_2(t) ) < (2 \sqrt{\rho_1(\Delta)} + 4 l_F) (2 \sqrt{\rho_2(\Delta)} + 4l_F)
 \end{multline*}
which contradicts \eqref{eqn:add-condDelta}. This proves the proposition.
\end{proof}

Let us now aim towards a computational classification of all real quadratic fields $F$ with class number 1 and discriminant $\leq 200$, $\neq 193,$ for which there is an $\of{}$-lattice with a universal lift to a quadratic extension $K/F$.

\begin{thm}\label{thm:table}
Table $\ref{table1}$ provides the complete list of real quadratic fields $F=\q(\sqrt{D_F})$ and totally real quadratic extensions $K/F$ such that there is an $\of{}$-lattice that is universal over $\ok{}$, among all fundamental discriminants $D_F \le 200, D_F\neq 193$, such that $F$ has class number $1$. For all of them, $D_F\leq 56$.

\begin{table}[h]
\caption{Complete list of $(D_F,K)$ for $D_F\le 200, D_F\neq 193$}
\label{table1}
\begin{tabular}{c|l}
\hline
$D_F$  & $K$ (LMFDB labels \cite{LMFDB})               \\ \hline
$5$  & $4.4.725.1$                          \\ \hline
$8$  & $4.4.1600.1$                         \\ \hline
$12$ & $4.4.2304.1$, $4.4.3600.1$, $4.4.4752.1$ \\ \hline
$17$ & $4.4.4913.1$                         \\ \hline
$21$ & $4.4.11025.1$                        \\ \hline
$24$ & $4.4.2304.1$, $4.4.14400.1$            \\ \hline
$28$ & $4.4.7056.1$, $4.4.19600.1$            \\ \hline
$29$ & $4.4.4205.1$                         \\ \hline
$33$ & $4.4.13068.1$                        \\ \hline
$56$ & $4.4.28224.1$                        \\ \hline
\end{tabular}
\end{table}

\end{thm}

To show the theorem, we proceed as follows: 
Let $F = \q(\sqrt{D})$ be a real quadratic field where $D>0$ is a fundamental discriminant. Let us set $U_F= \left\{0,1,\tau,1+\tau \right\}$, where $\of{} = \z + \z \tau$ with
\[
\tau=\begin{cases}
	\frac{1+\sqrt{D}}{2} & \text{if }D\equiv1\Mod{4},\\
	\frac{\sqrt{D}}{2} & \text{otherwise}.
\end{cases}
\]
Let $\epsilon$ be the square unit in $\{\delta\in\of{\times}^2: \rho_1(\delta)>\rho_2(\delta)\}$ that has the smallest value $\rho_1(\epsilon)$. Let us put $\gamma:=\sqrt{\rho_1(\epsilon)/\rho_2(\epsilon)}=\rho_1(\epsilon)$ and define
\[
\mathcal{F}:=\left\{\alpha \in F^+ : \frac{1}{\gamma} \le \frac{\rho_1(\alpha)}{\rho_2(\alpha)}<\gamma \right\}.
\]
Noting that $\epsilon$ is a generator of $\of{\times}^2$, one may observe that $\mathcal{F}$ is a fundamental domain of $F^+$ with respect to the action of multiplication by squares of units.

Let $K=F(\sqrt{\Delta})$ be a totally real number field containing $F$ such that there is an $\of{}$-lattice $L$ such that $L\otimes\ok{}$ is universal. We first utilize Propositions \ref{prop:quad_general_bound} and \ref{prop:additional-bound-Delta} to obtain a finite set $\mathcal{S}_1$ of elements in $\of{+}$ that $\Delta$ might belong to as follows.

By Proposition \ref{prop:quad_general_bound}, we have $K=F(\sqrt{\Delta})$ for some $\Delta\in\of{+}$. As the choice of $\Delta$ is unique up to multiplication by $\of{\times}^2$, we may assume that $\Delta\in\mathcal{F}$. Moreover, $\Delta$ should not satisfy the conditions in \eqref{eqn:condDelta} by Proposition \ref{prop:quad_general_bound} (4). We may conclude that $\Delta$ belongs to a bounded region $\mathcal{F}\cap \mathcal{B}_{c_{\max}}$ where 
\[
	\mathcal{B}_c:=\{\alpha\in F^+ : \rho_1(\alpha)<c \text{ or } \rho_2(\alpha)<c\}
\] 
for $c\in\mathbb{R}^+$ and 
\[
	c_{\max}:=\max\left(\{\rho_i(u^2) : u\in U_F,\ i\in\{1,2\}\}\cup\{\rho_i((2-u)^2) : u\in U_F,\ i\in\{1,2\} \}\cup\{9\}\right).
\] 
Therefore, we have only finitely many candidates for $\Delta$, and hence only finitely many candidates for $K$. Moreover, note from \eqref{eqn:condw_Delta} that $\Delta\equiv u^2 \Mod{4\of{}}$ for some $u\in U_F$. Furthermore, $\Delta$ should not satisfy the condition in \eqref{eqn:add-condDelta} by Proposition \ref{prop:additional-bound-Delta}. 
It turns out that, \eqref{eqn:add-condDelta} rules out significant number of elements in $\of{+}\cap \mathcal{F}\cap\mathcal{B}_{c_{\max}}$, especially when $\gamma$ gets too large. We use a computer program to obtain the finite set $\mathcal{S}_1$ of all elements $\alpha\in\of{+}$ contained in $\mathcal{F}\cap \mathcal{B}_{c_{\max}}$ such that $\alpha\equiv u^2 \Mod{4\of{}}$ for some $u\in U_F$ and not satisfying the condition \eqref{eqn:add-condDelta} with $\Delta=\alpha$. 

Let $\Delta\in\mathcal{S}_1$. We may assign a unique $w_\Delta=\frac{t+\sqrt{\Delta}}{2}\in\ok{}$ satisfying \eqref{eqn:condw_Delta} since $u^2$ are distinct modulo $4\of{}$ for all $u\in U_F$. Let us consider the following test to check for $m\in\of{}$ such that $\alpha_m:=m+w_\Delta\succeq0$, whether $\alpha_m$ is $F$-representable or not. Note that by Lemma \ref{lem:solve-pqr}, $\alpha_m$ is $F$-representable if and only if there is a solution $p,q,r\in\of{}$ satisfying $p,r,4pr-q^2\succeq 0$ to the equation
\begin{equation*}
	p + \frac{\Delta - t^2}{4}r = m \quad \text{and} \quad q + tr = 1.
\end{equation*}
Moreover, Lemma \ref{lem:solve-pqr} give us the following bound on $r$: 
\[
0\preceq r\preceq  \frac{4m+2t}{\Delta}
\]
We may search for every such $r\in\of{}$ and check whether $p=m-\frac{\Delta - t^2}{4}r\succeq0$ and $4pr-q^2\succeq0$. If this holds for some $r\in\of{}$ in the bound, then $\alpha_m$ is $F$-representable; otherwise, $\alpha_m$ is not $F$-representable.

For each $\Delta\in\mathcal{S}_1$, we go through this test for several $m\in\of{}$ such that $m\succeq -w_\Delta$, in the increasing order of $\tr_{F/\q}(m)$. If $\alpha_m=m+w_\Delta$ is not $F$-representable for one of those $m\in\of{}$, then $K=F(\sqrt{\Delta})$ does not have an $\of{}$-lattice such that $L\otimes\ok{}$ is universal. We collect those $\Delta\in\mathcal{S}_1$ such that all $\alpha_m$ are $F$-representable as a set $\mathcal{S}_2$. We note that a set $\mathcal{S}_2$ may be different depending on how many $m$ we test on.

Now for each $\Delta\in\mathcal{S}_2$, we try to prove that $K=F(\sqrt{\Delta})$ has an $\of{}$-lattice $L$ such that $L\otimes K$ is universal. By Theorem \ref{thm:equiv-existence-of-univ}, it suffice to show that every indecomposable element in $\ok{+}$ is $F$-representable. Furthermore, note that $\alpha$ is $F$-representable if and only if so are $\epsilon\alpha$ and $\zeta^2 \alpha$ for any unit $\epsilon\in\of{\times,+}$ and $\zeta \in \ok{\times}$. Hence it is enough to show that every representative of indecomposable element up to multiplication by $\of{\times,+} \ok{\times}^2$ is $F$-representable.

We seek to find a set of representatives of indecomposable elements up to multiplication by $\of{\times,+}\ok{\times}^2$. As $K$ is a totally real quartic field, it can't have cyclotomic units except $1$ and $-1$, so by Dirichlet's unit theorem $\ok{\times}$ is isomorphic to $\Z^3 \times \{-1,1\}$. It follows easily that $\ok{\times}^2 \simeq \Z^3$. Meanwhile, $\ok{\times,+}$ contains $\ok{\times}^2$ so $\ok{\times,+}$ is of rank $3$, and it does not contain any torsion. Thus $\ok{\times,+} \simeq \Z^3$. Since
\[
\ok{\times}^2 \subseteq \of{\times,+}\ok{\times}^2 \subseteq \ok{\times,+}
\]
and they are groups, we have $\of{\times,+}\ok{\times}^2 \simeq \Z^3$. Let $u_1, u_2, u_3$ be a basis of $\of{\times,+}\ok{\times}^2$.

Let $\alpha$ be an indecomposable element in $\ok{+}$. By multiplying by an element in $\of{\times,+}\ok{\times}^2$ if necessary, we may assume that $\alpha$ is contained in the cone generated by elements $\mathfrak{U}=\{1,u_1,u_2,u_3,u_1u_2,$ $u_1u_3, u_2u_3,u_1u_2u_3\}$ (see \cite[Theorem 1]{DF}). Noting that $\alpha-u\notin\ok{+}$ for any $u \in \mathfrak{U}$ since $\alpha$ is indecomposable, $\alpha$ is contained in the $4$-dimensional cube $\mathcal{C}:=\left\{\sum_{i=1}^4 a_iv_i : a_i\in[0,1]\right\}$ generated by $4$ elements $v_1,v_2,v_3,v_4$ in $\mathfrak{U}$. Writing $\alpha=\sum_{i=1}^4 a_iv_i$ for some $a_i\in[0,1]$, we have
\[
	\tr_{K/\q}(\alpha)=\sum_{i=1}^4 a_i\tr_{K/\q}(v_i)\le \sum_{i=1}^4 \tr_{K/\q}(v_i)\le M(\mathfrak{U}):=\max \left\{\sum_{i=1}^4 \tr_{K/\q}(v_i) : v_i \in \mathfrak{U}\right\}.
\]
Therefore, we may list all elements $\alpha\in\ok{+}$ with $\tr_{K/\q}(\alpha)\le M(\mathfrak{U})$ and then find the set $S_{\text{indec}}$ of indecomposable elements among this list. Notice that the indecomposable elements in $S_{\text{indec}}$ found so far may coincide up to a unit multiplication. For each $\alpha=u+vw_\Delta\in S_{\text{indec}}$ with $u,v\in \of{}$, we make use of Lemma \ref{lem:solve-pqr} to check if $\alpha$ is $F$-representable, namely, we check if there are $p,q,r \in \of{}$ such that $p,r,4pr-q^2\succeq 0$ and satisfying
\[
	p + \frac{\Delta - t^2}{4}r = u \quad \text{and} \quad q + tr = v.
\]
This can be done by trying all $r\in\of{}$ such that $0\preceq r \preceq (4u+2vt)/\Delta$. If every $\alpha\in S_{\text{indec}}$ is $F$-representable, then we may conclude that there is an $\of{}$-lattice $L$ such that $L\otimes \ok{}$ is universal; otherwise, there is not.

\begin{rmk}\label{rmk:code}
    We carried out all the computations for proving Theorem \ref{thm:table} in Mathematica. We covered each value of $D_F$ separately using a series of small programs (available upon request) together with data from LMFDB \cite{LMFDB}. The run-time significantly depends on the size of the region $\mathcal B_c$ that, in turn, depends on the size $\gamma$ of the fundamental unit. 
    
    Thus it was fast to run for small values $D_F<125$ (less than 1 minute each), it took approx. 7 hours for $D_F=177$, and did not complete within 1 day for the excluded case $D_F=193$. In this case, one would have to check elements with traces up to $8.5$ billion as candidates for the discriminant; in the runtime of 1 day we got to trace $0.3$ billion. While it should be possible to finish the check for $193$ and some values $D_F>200$, it would be difficult to reach significantly larger $D_F$, at least without non-trivial improvements in the algorithm and its implementation.
\end{rmk}

Let us further illustrate the algorithm on three explicit examples.

\begin{exam} Let us consider the case when $D=5$, that is, $F=\q(\sqrt{5})$. Note that $\of{} = \z + \z \tau$ with $\tau=\frac{1+\sqrt{5}}{2}$. Moreover, we have $\epsilon=\left(\frac{1+\sqrt{5}}{2}\right)^2=\frac{3+\sqrt{5}}{2}$, hence $\gamma=\frac{1+\sqrt{5}}{2}$. Proceeding the above steps, we have
\[
	\mathcal{S}_2=\left\{5,\frac{11+\sqrt{5}}{2},\frac{11-\sqrt{5}}{2}\right\}.
\]

If $\Delta=5$, then $K=F(\sqrt{5})=F$, hence this is not the case we are interested in.
If $\Delta=\frac{11\pm\sqrt{5}}{2}$, then $\Delta_{K/\q}=\Delta_{F/\q}^2\cdot  N_{F/\q}(\Delta)=5^2\cdot 29=725$. Searching in LMFDB, there are exactly one totally real quartic field containing $F=\q\left(\sqrt{5}\right)$ with $\Delta_{K/\q}=725$, which is labelled 4.4.725.1. For this field $K$, we could check that every indecomposable element in $\ok{+}$ is $F$-representable.
\end{exam}

\begin{exam}
Let us consider the case when $D=8$, that is, when $F=\q(\sqrt{2})$. In this case, $\mathcal{S}_2=\{5\}$ and for $\Delta=5$, $K=F(\sqrt{\Delta})=\q(\sqrt{2},\sqrt{5})$. For this field $K$, we could check that every indecomposable element in $\ok{+}$ is $F$-representable. Thus, for this $K$, there is an $\of{}$-lattice that is universal over $\ok{}$.
\end{exam}

\begin{exam}
Let us consider the case when $D=12$, that is, when $F=\q(\sqrt{3})$. 
In this case, a totally real quartic field $K$ containing $F$ that has an $\of{}$-lattice $L$ such that $L\otimes K$ is universal if and only if $K$ is isomorphic to 
\[
	\q(\sqrt{2},\sqrt{3}), \ \q(\sqrt{3},\sqrt{5}) \text{ and } \q\left(\sqrt{3},\sqrt{9 + 4 \sqrt{3}}\right).
\]
Note that they correspond to LMFDB labels 4.4.2304.1, 4.4.3600.1 and 4.4.4752.1.
\end{exam}

\begin{conj}\label{conj:real quad lifts}
	Let $F=\Q(\sqrt{D})$ be a real quadratic field where $D$ is the discriminant of $F$. If $D>56$ and $F$ has class number $1$, then there is no totally real quartic field $K$ containing $F$ which has an $\of{}$-lattice that is universal over $\ok{}$.
\end{conj}

\section{Extension by fixed $\sqrt{e}$}\label{sec:5}

\begin{lem}\label{lem:field-compositum}
Let $F_1, F_2$ be fields of degree $n_1, n_2$ respectively such that $\mathrm{gcd}(\mathrm{Disc}(F_1), \mathrm{Disc}(F_2))= 1$, and suppose that $F_1 F_2$ has degree $n_1 n_2$. Then its discriminant is $\mathrm{Disc}(F_1)^{n_2} \mathrm{Disc}(F_2)^{n_1}$, and $\oo_{F_1 F_2} = \oo_{F_1} \oo_{F_2}$.
\end{lem}

\begin{proof}
Note that we have $\mathrm{Disc}(\oo_{F_1} \oo_{F_2}) = \mathrm{Disc}(F_1)^{n_2} \mathrm{Disc}(F_2)^{n_1}$ and $\oo_{F_1} \oo_{F_2} \subseteq \oo_{F_1 F_2}$. Hence $\mathrm{Disc}(F_1 F_2)$ divides $\mathrm{Disc}(F_1)^{n_2} \mathrm{Disc}(F_2)^{n_1}$. 
Meanwhile, the relative discriminant formula 
\[
\mathrm{Disc}(F_1 F_2) = \mathrm{Disc}(F_1)^{n_2} N(\Delta_{F_1 F_2/F_1})
\]
implies $\mathrm{Disc}(F_1)^{n_2} \vert \mathrm{Disc}(F_1 F_2)$ and similarly $\mathrm{Disc}(F_2)^{n_1} \vert \mathrm{Disc}(F_1 F_2)$. As $\mathrm{gcd}(\mathrm{Disc}(F_1), \mathrm{Disc}(F_2)) =1$, we have $\mathrm{Disc}(F_1)^{n_2} \mathrm{Disc}(F_2)^{n_1} \vert \mathrm{Disc}(F_1 F_2)$. So it follows that $\mathrm{Disc}(F_1 F_2) = \mathrm{Disc}(F_1)^{n_2} \mathrm{Disc}(F_2)^{n_1} = \mathrm{Disc}(\oo_{F_1} \oo_{F_2})$. Thus the inclusion $\oo_{F_1} \oo_{F_2} \subseteq \oo_{F_1 F_2}$ should be equality.
\end{proof}

\begin{lem}\label{lem:extn-by-sqrt(d)}
Let $e$ be a square-free integer and let $F$ be a field such that $\mathrm{gcd}\left(\mathrm{Disc}(F), D_{\q(\sqrt{e})}\right)=1$. Then the field $K=F(\sqrt{e})$ is a quadratic extension of $F$ and $\oo_{K} = \oo_{F} 1 + \oo_{F} \omega_e$ where $\omega_e = \frac{1+\sqrt{e}}{2}$ if $e \equiv 1 \Mod 4$ and $\omega_e =\sqrt{e}$ otherwise.
\end{lem}

\begin{proof}
If $K/F$ is not quadratic, then $\sqrt{e} \in F$, so $\q(\sqrt{e}) \subset F$. Noting that $\mathrm{Disc}(\q(\sqrt{e}))=\Delta$, we may observe that $\mathrm{gcd}(\mathrm{Disc}(F), \Delta)=1$ cannot happen. Thus by Lemma \ref{lem:field-compositum} 
we have $\oo_{K} = \oo_{F} 1 + \oo_{F} \omega_e$.
\end{proof}

\begin{thm}\label{prop:lifting-fail-sqrt(d)-extn}
Let $F$ be a totally real field, $e>0$ a square-free integer such that $\mathrm{gcd}(\mathrm{Disc}(F), \Delta)=1$, where $\Delta=D_{\q(\sqrt{e})}$; then  $K=F(\sqrt{e})=F(\sqrt{\Delta})$ is a totally real quadratic extension of $F$. If $e\neq 5$, then there is no $\of{}$-lattice $L$ such that $L\otimes\ok{}$ is universal over $\ok{}$.
\end{thm}

\begin{proof} First, let us consider the case when $e\equiv 2,3\Mod{4}$. Note that $\ok{} = \of{} + \of{} \sqrt{e}$ by Lemma \ref{lem:extn-by-sqrt(d)}. Let $n_e = 1 + \lfloor \sqrt{e} \rfloor$ and consider the element $\alpha = n_e + \sqrt{e}\in\ok{+}$. We may apply Lemma \ref{lem:solve-pqr} with $\omega_\Delta=\frac{0+\sqrt{\Delta}}{2}=\sqrt{e}$ to show that $\alpha$ is $F$-representable if and only if $p + er = n_e$ and $q=1$ for some $p,r \in \of{+}$ satisfying $4pr-q^2 \succeq 0$. Since $e \ge n_e$, we should have either $r=0$ or $p=0$, which is a contradiction as $4pr-q^2=-1\prec0$.

Now let us consider the case when $e\equiv 1\Mod{4}$ with $e>5$. Let $\omega_e = \frac{1+\sqrt{e}}{2}$ and $n_e =  \lfloor \omega_e \rfloor$. One may show that $\alpha=n_e + \omega_e\in\ok{+}$. Applying Lemma \ref{lem:solve-pqr} with $\omega_e$, we may verify that $\alpha$ is $F$-presentable if and only if $p+\frac{e-1}{4} r = n_e$ and $q+r=1$ for some $p,q,r \in \of{}$ satisfying $p,r,4pr-q^2 \succeq 0$. Noting that $e \ge 13$, we have $\frac{e-1}{4}  = \frac{\sqrt{e}+1}{2} \frac{\sqrt{e}-1}{2} > \lfloor \frac{\sqrt{e}+1}{2} \rfloor \frac{\sqrt{13}-1}{2}> n_e$. Hence we should have $r=0$, and hence $q=1$, which is a contradiction $4pr-q^2=-1\prec0$.
\end{proof}

\begin{rmk}
Note that the proof of Theorem \ref{prop:lifting-fail-sqrt(d)-extn} does not work when $e=5$ since $\alpha = 1+\frac{1+\sqrt{5}}{2}=\left(\frac{1+\sqrt{5}}{2}\right)^2$ is $F$-representable.
\end{rmk}

\subsection{Extension by $\sqrt{5}$}
Motivated by Theorem \ref{prop:lifting-fail-sqrt(d)-extn}, we are interested in investigating totally real number field $F$ with $\mathrm{gcd}(\mathrm{Disc}(F),5)=1$ which admits an $\of{}$-lattice that is universal over $F(\sqrt{5})$. 

Let $K=F(\sqrt{5})$ and let $\epsilon = \frac{1+\sqrt{5}}{2}$. By Lemma \ref{lem:extn-by-sqrt(d)}, we have $\ok{} = \of{} + \of{}\epsilon$. We first introduce a lemma describing conditions for an element $a+b\epsilon\in\ok{}$ with $a,b\in\of{}$ to be totally positive and $F$-representable.
\begin{lem}\label{lem:extn-by-sqrt5}
   Let $K=F(\sqrt{5})$ and let $\alpha=a+b\epsilon\in\ok{}$ with $a,b\in\of{}$. Then we have the following:
   \begin{enumerate}[label={{\rm(\arabic*)}}]
       \item $\alpha \in\ok{+}$ if and only if $a\in\of{+}$ and $(1-\epsilon)\rho(a)< \rho(b) < \epsilon \rho(a)$ for any embedding $\rho:F\ra \mathbb{R}$.
       \item $\alpha\in\ok{+}$ is $F$-representable if and only if there exist $p,q,r \in \of{}$ with $p, r, 4pr-q^2 \succeq 0$ satisfying
\[
a=p+r \quad  \text{and} \quad b=q+r.
\]
       If this is the case, then $0\preceq r \preceq a$ in $F$ and for any embedding $\rho:F\ra \mathbb{R}$ with $\beta = \rho\left(\frac{b}{a}\right)$, we have
\begin{equation}\label{eqn:rhor/a-bound}
\frac{\beta+2}{5} - \frac{2}{5} \sqrt{1+ \beta - \beta^2} \le \rho\left(\frac{r}{a}\right) \le \frac{\beta+2}{5} + \frac{2}{5} \sqrt{1+ \beta - \beta^2}.
\end{equation}

    \end{enumerate}
\end{lem}
\begin{proof}
To show part (1), note that $\alpha\in\ok{+}$ if and only if $\rho(a)+\rho(b)\epsilon>0$ and $\rho(a)+\rho(b)(1-\epsilon)>0$ for any embedding $\rho:F\ra \mathbb{R}$. Thus the inequality $(1-\epsilon)\rho(a)< \rho(b) < \epsilon \rho(a)$ follows immediately. Noting that $\epsilon>0$ while $1-\epsilon<0$, we should have $\rho(a)>0$, since otherwise, we have either $\rho(a)+\rho(b)\epsilon\le0$ or $\rho(a)+\rho(b)(1-\epsilon)\le0$. Hence we have $a\in\of{+}$. 

Part (2) of the lemma follows immediately from Lemma \ref{lem:solve-pqr} and \eqref{eqn:rbound} with the fact that $\rho(a)>0$. 
\end{proof}

\begin{thm}\label{thm:extn-by-sqrt5}
Let $F$ be a real quadratic field such that $\gcd(D_F,5)=1$; then $K=F(\sqrt{5})$ is a real quadratic extension of $F$. Suppose that there is an $\of{}$-lattice $L$ such that $L\otimes \ok{}$ is universal over $\ok{}$. Then $D_F\le 4076$, and hence there are only finite many such real quadratic fields $F$. 
\end{thm}

\begin{proof}

Recall that $\of{}=\z 1 + \z \tau$ where $\tau\in\of{}$ is defined as in \eqref{eqn:def-O_F} and let $\rho_1=\mathrm{id}_F$, $\rho_2$ be two embeddings of $F$ into $\mathbb{R}$. Assuming that $D=D_F\ge4077$, we will construct $a\in\of{+}$ and $b\in \of{}$ such that $a+b\epsilon \in \ok{+}$ which is not $F$-representable. The existence of such an element $a+b\epsilon$ immediately implies the theorem. 

Assume that $D=D_F\ge4077$. We first take $a\in\of{}$ to be the unique element in $\tau+\z 1 \subseteq \of{}$ satisfying
\begin{equation}\label{eqn:fa-bound}
2\sqrt{D}+1 < f(a):=\epsilon \rho_1(a) + (\epsilon -1 )\rho_2(a) \le 2\sqrt{D} + 1 + \sqrt{5}.
\end{equation}
Such an element exists uniquely since for $a'\in\tau+\z1$, $f(a'+1)=f(a')+2\epsilon-1=f(a')+\sqrt{5}$. Note that $\rho_1(a)=\rho_2(a)+\sqrt{D}$, and hence $f(a)=\epsilon\sqrt{D}+\sqrt{5}\rho_2(a)>2\sqrt{D}+1$ implies that $\rho_2(a)>0$. Moreover, $f(a)=\sqrt{5}\rho_1(a)+\frac{1-\sqrt{5}}{2}\sqrt{D}$ with \eqref{eqn:fa-bound} yields the bound of
\begin{equation}\label{eqn:rho1a-bound}
    \frac{1}{\sqrt{5}}\left(\frac{3+\sqrt{5}}{2}\sqrt{D}+1\right) < \rho_1(a)\le \frac{1}{\sqrt{5}}\left(\frac{3+\sqrt{5}}{2}\sqrt{D}+1+\sqrt{5}\right).
\end{equation}

Now we construct $b\in 2\tau + \z1\subset\of{}$ such that $\rho_1(b)<\epsilon\rho_1(a)$ and $(1-\epsilon)\rho_2(a)<\rho_2(b)$. Consider an injective map $\widetilde{\rho}:\of{}\ra \mathbb{R}^2$ defined by $\widetilde{\rho}(\alpha)=(\rho_1(\alpha),\rho_2(\alpha))$. Then all the points $\widetilde{\rho}(2\tau+\z1)$ lies on the line $L=\{(x,y)\in\mathbb{R}^2 : x - y = 2 \sqrt{D} \}$, placed in every $\sqrt{2}$ distance. Note that the point $P=(\epsilon \rho_1(a), (1-\epsilon) \rho_2(a))$ is on the line $L'=\{(x,y)\in\mathbb{R}^2 : x - y = f(a)\}$. Since $f(a)>2\sqrt{D}+1$ from the construction, $P$ is located below $L$ with distance $>1/\sqrt{2}$. Hence the segment formed by the line $L$ and the intersection of region $\{(x,y)\in\mathbb{R}^2 : x<\epsilon \rho_1(a), y > (1-\epsilon)\rho_2(a) \}$ is of length $>\sqrt{2}$. Thus this segment contains a point $P_0$ in $\widetilde{\rho}(2\tau + \z1)$, whose inverse $\widetilde{\rho}^{-1}(P_0)$ is an element $b$ we were looking for.

Next step is to show that $a+b\epsilon\in\ok{+}$. By Lemma \ref{lem:extn-by-sqrt5} (1) and since $\rho_1(b)<\epsilon\rho_1(a)$ and $(1-\epsilon)\rho_2(a)<\rho_2(b)$ from the construction, it suffices to show that $(1-\epsilon)\rho_1(a)<\rho_1(b)$ and $\rho_2(b)<\epsilon\rho_2(a)$. Note that
\begin{equation}\label{eqn:rho1b-bound-lower}
    \rho_1(b) = 2\sqrt{D} + \rho_2(b) > 2 \sqrt{D} + (1- \epsilon) \rho_2(a)\ge  f(a) - (1+\sqrt{5})+  (1-\epsilon) \rho_2(a) = \epsilon \rho_1(a) - (1+\sqrt{5}),
\end{equation}
where we used \eqref{eqn:fa-bound} in the last two steps. Using \eqref{eqn:rho1a-bound}, we may verify that $\epsilon\rho_1(a)-(1+\sqrt{5}) > (1-\epsilon)\rho_1(a)$. Similarly, one may use \eqref{eqn:fa-bound} and \eqref{eqn:rho1a-bound} to obtain $\rho_2(b)<\epsilon\rho_2(a)$.

We now show that $a+b\epsilon$ is not $F$-representable. Assume to the contrary that $a+b\epsilon$ is $F$-representable. Then by Lemma \ref{lem:extn-by-sqrt5} (2), there is $r\in\of{}$ with $0\preceq r \preceq a$ and satisfying \eqref{eqn:rhor/a-bound}. If $\rho_2(a)<\rho_1(r)<\rho_1(a)-\rho_2(a)$, then since $0<\rho_2(r)<\rho_2(a)$ we have $0<\rho_1(r)-\rho_2(r)<\rho_1(a)-\rho_2(a)=\sqrt{D}$. However, this is impossible since $\rho_1(\alpha)-\rho_2(\alpha)\in \z\sqrt{D}$ for any $\alpha\in\of{}$. Thus we have $\rho_1(r)\le\rho_2(a)$ or $\rho_1(r)\ge \rho_1(a)-\rho_2(a)$. Combining this with $\rho_2(a)=\rho_1(a)-\sqrt{D}$ and \eqref{eqn:rho1a-bound}, we have
\[
	0 \le \rho_1\left(\frac{r}{a}\right) \le 1 - \frac{ 2 \sqrt{5}}{3 + \sqrt{5} + 2 D^{-1/2}} \quad \text{or} \quad \frac{ 2 \sqrt{5}}{3 + \sqrt{5} + 2 (1+\sqrt{5}) D^{-1/2}} \le \rho_1\left(\frac{r}{a}\right) \le 1.
\]
Since we are assuming $D\ge 4077$, we have
\begin{equation}\label{eqn:rho1r/a-bound1}
    0 \le \rho_1\left(\frac{r}{a}\right) \le 0.151 \quad \text{or} \quad 0.837881\le \rho_1\left(\frac{r}{a}\right) \le 1.
\end{equation}

On the other hand, putting $\beta=\rho_1\left(\frac{b}{a}\right)$, recalling the bound $\epsilon\rho_1(a)-(1+\sqrt{5})<\rho_1(b)<\epsilon\rho_1(a)$ from \eqref{eqn:rho1b-bound-lower} and combining those with \eqref{eqn:rho1a-bound} and $D\ge4077$, we have the bound
\[
    \epsilon-\frac{10+2\sqrt{5}}{(3+\sqrt{5})\sqrt{4077}+2}\le\epsilon-\frac{10+2\sqrt{5}}{(3+\sqrt{5})\sqrt{D}+2}<\epsilon-\frac{1+\sqrt{5}}{\rho_1(a)}<\beta< \epsilon.
\]
In this range of $\beta$, the left-hand-side of \eqref{eqn:rhor/a-bound} is increasing while the right-hand-side of \eqref{eqn:rhor/a-bound} is decreasing, yielding the bound
\begin{equation}\label{eqn:rho1r/a-bound2}
    0.592\le \rho_1\left(\frac{r}{a}\right) \le 0.837877
\end{equation}
Since two bounds \eqref{eqn:rho1r/a-bound1} and \eqref{eqn:rho1r/a-bound2} are disjoint, we get a contradiction, which means that $a+b\epsilon$ is not $F$-representable.
\end{proof}

\end{document}